\documentclass[12pt]{amsart}
\usepackage{amssymb}
\usepackage{amsmath}
\usepackage{amsthm}
\usepackage{times}
\usepackage{graphicx}

\usepackage[curve,all]{xy}
 \xyoption{curve}
 \xyoption{line}
\xyoption{arc}

\usepackage{color}
\usepackage{amsthm}

\newtheorem{theorem}{Theorem}[section]
\newtheorem{lemma}[theorem]{Lemma}
\newtheorem{corollary}[theorem]{Corollary}
\newtheorem{proposition}[theorem]{Proposition}

\theoremstyle{definition}

\newtheorem{example}[theorem]{Example}

\theoremstyle{remark}
\newtheorem{remark}[theorem]{Remark}

\numberwithin{equation}{section}

\newcommand{\bbF}{\mathbb{F}}
\newcommand{\bbC}{\mathbb{C}}

\newcommand{\bbP}{\mathbb{P}}

\newcommand{\bbQ}{\mathbb{Q}}

\newcommand{\bbZ}{\mathbb{Z}}

\def\det{{\text{det}}}

\def\NS{{\rm{NS}}}

\def\spmapright#1{\smash{%
   \mathop{\hbox to 1.3cm{\rightarrowfill}}
       \limits^{#1}}}
\def\sbmapright#1{\smash{%
   \mathop{\hbox to 1.3cm{\rightarrowfill}}
       \limits_{#1}}}

\newcommand{\mapright}[1]{%
\smash{\mathop{%
   \hbox to 1cm{\rightarrowfill}}\limits^{#1}}}
\newcommand{\mapleft}[1]{%
\smash{\mathop{%
   \hbox to 1cm{\leftarrowfill}}\limits^{#1}}}

\begin{document}
\title[K3 surfaces with 9 cusps]{K3 surfaces with 9 cusps in characteristic $p$}
\author{Toshiyuki Katsura}
\thanks{Partially supported by JSPS Grant-in-Aid 
for Scientific Research (B) No. 15H03614} 
\address{Faculty of Science and Engineering, Hosei University,
Koganei-shi, Tokyo 184-8584, Japan}
\email{toshiyuki.katsura.tk@hosei.ac.jp}

\author{Matthias Sch\"utt}
\address{Institut f\"ur Algebraische Geometrie, 
Leibniz Universit\"at  Hannover, Welfengarten 1, 30167 Hannover, Germany, and
\newline\indent
Riemann Center for Geometry and Physics, 
  Leibniz Universit\"at Hannover, 
  Appelstrasse 2, 30167 Hannover, Germany}
\email{schuett@math.uni-hannover.de}

\date{February 1, 2019}

\begin{abstract}
We study K3 surfaces with 9 cusps, i.e.\ 9 disjoint $A_2$ configurations
of smooth rational curves, over algebraically closed fields of characteristic $p\neq 3$.
Much like in the complex situation studied by Barth,
we prove that each such surface admits a triple covering by an abelian surface.
Conversely, we determine which  abelian surfaces with order three automorphisms
give rise to K3 surfaces.
We also investigate how K3 surfaces with 9 cusps hit the supersingular locus.
\end{abstract}

\maketitle

\section{Introduction}

In two papers from the 1990's \cite{Barth1}, \cite{Barth2},
Barth studied complex K3 surfaces with 9 cusps,
i.e.\ with 9 disjoint $A_2$ configurations of smooth rational curves.
Barth's arguments were of topological nature,
using a triple cover by some suitable abelian surface.
In this paper, we follow a more algebraic approach
which lends itself to investigate the same problem
over any algebraically closed field $k$ of characteristic $p\neq 3$
(which we fix throughout this paper).
This enables us to detect several interesting phenomena;
in particular, we also include the Zariski K3 surfaces in characteristics $p\equiv -1$ mod $3$
from \cite{KS}.
Combined with explicit calculations for abelian surfaces (in positive characteristic)
and the characteristic-free divisibility results for certain divisor classes from \cite{S-nodal},
we prove the following results:
 
\begin{theorem}
\label{thm1}
If $X$ is a K3 surface with 9 cusps, then $X$ admits a triple covering by an abelian surface
with an automorphism of order $3$.
\end{theorem}

\begin{theorem}
\label{thm2}
If $X$ is a supersingular K3 surface with 9 cusps, then 
\begin{itemize}
\item
either $X$ is the supersingular K3 surface of Artin invariant $\sigma=1$,
\item
or $X$ has Artin invariant $\sigma=2$ and $p\equiv -1\;  {\rm mod } \; 3$.
\end{itemize}
\end{theorem}

Both theorems are supported by ample examples,
starting from suitable abelian surfaces with an automorphism of order $3$.
In fact, for an abelian surface to admit a triple K3 quotient is quite restrictive,
both in the simple and non-simple case:

\begin{proposition}
\label{non-isogenous}
Let $A$ be an abelian surface such that A is isogenous to $E_{1}\times E_{2}$
with elliptic curves $E_{i}$ ($i = 1, 2$). Assume $E_{1}$ is not isogenous to $E_{2}$.
Let $\sigma$ be an automorphism of $A$ of order 3. Then, $A/\langle \sigma \rangle$
is not birationally equivalent to a K3 surface.
\end{proposition}

In comparison, simple abelian surfaces are quite delicate to treat 
as we shall explore in Sections \ref{s:endo}, \ref{s:simple}.
In the context of this paper, it turns out that ordinarity enters as an essential ingredient:

\begin{theorem}
\label{thm:simple}
Let $A$ be a simple ordinary abelian surface with an automorphism
$\sigma$ of order 3. Assume 
that $\sigma$ is not a translation.
Then, the quotient surface $A/ \langle \sigma \rangle$ is birationally equivalent 
to a K3 surface.
\end{theorem}

Simple abelian surfaces in positive characteristic
turn out to be quite hard to exhibit explicitly, especially in characteristic $2$.
Therefore we conclude the paper  with an explicit one-dimensional family of K3 surfaces
with 9 cusps valid in any characteristic $\neq 3, 5$
such that the generic covering abelian surface is simple
(and we also provide an alternative family covering characteristic $5$).

\begin{remark}
Many of our arguments also work in characteristic zero,
but to ease the presentation we decided to restrict to the positive characteristic case.
\end{remark}

\section{Lattice theory for K3 surfaces with 9 cusps}

Let $X$ be an algebraic K3 surface over an algebraically closed field $k$ of characteristic $p\neq 3$.
Assume that $X$ contains 9 disjoint $A_2$ configurations of smooth rational curves. Then
we have to determine the primitive closure of the resulting sublattice in $\NS(X)$:
\[
L := (A_2^9)' \subset\NS(X).
\]
From general lattice theory (see e.g.\ \cite{Nikulin}),
we know that $L$ is determined by some isotropic subgroup $H$
of the discriminant group $G=(A_2^\vee)^9/A_2^9$.
Here the latter space is identified with the vector space $\bbF_3^9$,
so the given problem can be analysed using coding theory.
In \cite{Barth2}, Barth achieves this by showing (topologically over $\bbC$)
the following lemma:

\begin{lemma}
\label{lem:69}
Any non-zero  vector in $H$ has length 6 or 9.
\end{lemma}

\begin{proof}
The same holds true in arbitrary characteristic  since
isotropic vectors of length $3$ would yield a  vector in $L\setminus A_2^9$ of square $-2$,
a contradiction to \cite{S-nodal} (which is valid in any characteristic).
\end{proof}

In order to determine $L$, it will be instrumental to work out a suitable reference lattice $\Lambda$
into which $L$ embeds primitively.
If $X$ has finite height, then it is known by work of Deligne 
that $X$ lifts to characteristic zero with a full set of generators of $\NS(X)$.
Hence we can take $\Lambda$ to be the standard  even unimodular lattice of rank $22$ and signature $(3,19)$,
\[
\Lambda = \Lambda_\text{K3} = U^3 + E_8^2
\]
(to which $H^2(Y,\bbZ)$ of any complex K3 surface $Y$ is isomorphic).
On the other hand, if $X$ is supersingular,
say of Artin invariant $\sigma$,
then we may just take
\[
\Lambda= \Lambda_{p,\sigma} = \NS(X)
\]
the unique even hyperbolic lattice of rank $22$ and discriminant group 
$$A_\Lambda = \bbF_p^{\,2\sigma}.
$$
What unifies both variants is that they have the same rank while being prime to $3$ in the sense that, by assumption,
the discriminant is not divisible by $3$.
In comparison, $L$ is also prime to $p$ since it has discriminant $-3^r$, where $r=9-2\cdot |H|$.
More precisely, $L$ has discriminant group 
\begin{eqnarray}
\label{eq:r}
A_L\cong \bbF_3^{\,r}.
\end{eqnarray}

By construction, $L$ embeds primitively into $\Lambda$.
Since $L$ and $\Lambda$ are relatively prime in the above terminology,
the orthogonal complement $L^\perp$ admits a subgroup $H\subseteq A_{L^\perp}$ such that 
not only the discriminant groups are isomorphic,
\begin{eqnarray}
\label{eq:H}
H \cong A_L,
\end{eqnarray}
but also the discriminant forms agree up to sign:
\[
q_L = - q_{L^\perp}|_H.
\]
In particular, $H$ and $A_L$ share the same length (i.e.\ minimum number of generators).
Presently this is $r$ by \eqref{eq:r}, and on the other hand,
the length is a priori bounded by the rank of $L^\perp$, i.e.  $r\leq 4$.

\begin{lemma}
\label{lem:L}
$L$ is an overlattice of $A_2^9$ of index $27$,
determined uniquely up to isometries by its discriminant form
\[
q_L = - q_M
\]
for $M=U(3)+A_2(-1)$.
\end{lemma}

\begin{proof}
The construction of $L$ follows exactly the lines of \cite{Barth2},
just using the existence of some isomorphism \eqref{eq:H} and Lemma \ref{lem:69}.
In particular, this shows that $L$ is unique up to isometries.
In loc. cit.  it was also proved that $L^\perp$ inside $\Lambda_\text{K3}$ 
is isometric to $M=U(3)+A_2(-1)$.
Since the shape of $L$ does not depend on the characteristic,
the statement on the discriminant forms is always valid.
\end{proof}

\begin{remark}
\label{rem:H}
The above argument also shows as in \cite{Barth2}
that the subgroup $H$ of $G$ contains a vector of length $9$.
This will be quite useful in the proof of Theorem \ref{thm1}.
\end{remark}

\section{Proof of Theorem \ref{thm2}}
\label{s:pf2}

We are now ready to prove Theorem \ref{thm2}.
To support it, we recall the following two well-known constructions
of K3 surfaces with 9 cusps (cf. Katsura \cite{K}, for instance).

\begin{example}
Let $E$ be an elliptic curve defined by
$$
y^{2} +y = x^{3},
$$
and let $\sigma$ be an automorphism of $E$ given by
$$
   x \mapsto \omega x,~ \;\; y \mapsto y
$$
with $\omega$ a primitive cube root of unity. Then, $\sigma \times \sigma^{2}$ is 
an automorphism
of the abelian surface $A = E \times E$ of order 3 and the quotient surface 
$A/\langle \sigma \times \sigma^{2} \rangle$ is birationally equivalent to a K3 surface
with 9 cusps.
Note that in case $p \equiv 1 ~({\rm mod} ~3)$, $A$ is ordinary, and 
in case $p \equiv -1 ~({\rm mod} ~3)$, $A$ is supersingular
(since the same holds for the elliptic curve $E$).
\end{example}

\begin{example}
Let $E$ be an elliptic curve, and we set $A = E \times E$. Let $\sigma$ be
the automorphism of $A$ defined by
$$
\left(
\begin{array}{cc}
   0 & \iota \\
   id & \iota
   \end{array}
   \right)
$$
where $\iota$ is the inversion of $E$.
Then $\sigma$  has order 3 and 
the quotient surface 
$A/\langle \sigma \rangle$ is birationally equivalent to a K3 surface with 9 cusps.
Note that in case $E$ is ordinary, $A$ is also ordinary, and 
in case $E$ is supersingular, $A$ is also supersingular.
\end{example}

%

Accidentally, we treated the case of supersingular K3 surfaces of Artin invariant $\sigma=2$
in characteristic $p\equiv -1$ mod $ 3$ in \cite{KS}.
Namely we proved that all these K3 surfaces are Zariski
(i.e.\ unirational, admitting an inseparable covering by $\bbP^2$ of degree $p$)
by exploiting exactly the structures imposed by a configuration of 9 disjoint $A_2$'s.

\medskip

In order to prove Theorem \ref{thm2},
it remains
to treat 
the cases of Artin invariants $\sigma>2$ as well as $\sigma=2$ in characteristic $p\equiv 1$ mod $3$.

\subsection{Artin invariant $\sigma>2$}

In the previous section, we bounded the length of $A_L$ by considering the orthogonal complement $L^\perp$
inside the reference lattice $\Lambda$ (which was coprime to $L$).
Here we can argue along similar lines for $\Lambda$ itself
(cf.\ \cite[Thm.\ 6.1]{KS}).
Namely, for the same reason as above,
the discriminant group $A_\Lambda\cong \bbF_p^{\, 2\sigma}$ has to be supported on $A_{L^\perp}$.
But again, $L^\perp$ has rank $4$, so $\sigma\leq 2$ as claimed.

\subsection{Artin invariant $\sigma=2$ in characteristic $p\equiv 1$ mod $3$}

For reasons to become clear in a moment,
we omit the restriction on the characteristic for the time being.
That is, we just assume that $\sigma=2$, and for simplicity that $p>3$
(because for computations with even lattices it is often easier to exclude $p=2$).
Suppose that $L$ admits a primitive embedding 
\[
L\hookrightarrow \Lambda = \Lambda_{p,2}
\]
and let $L^\perp$ denote the orthogonal complement as before.
We have seen above that $A_\Lambda$ is supported on $A_{L^\perp}$.
Presently, this means that $A_{L^\perp}$ has $p$-length $4$,
i.e.\  $L^\perp$ is $p$-divisible as an even lattice (since $p>2$).
We can thus scale $L^\perp$ by $1/p$ and obtain an even hyperbolic lattice
\begin{eqnarray}
\label{eq:N}
N = L^\perp\left(\frac 1p\right),
\end{eqnarray}
 of rank $4$ and discriminant $-27$ (the same as the discriminant of  $L$ up to sign).
We claim that
\[
N \cong U(3)+A_2.
\]
To see the claim, we impose a duality in the spirit of \cite{Kondo-Shimada}
to derive the even hyperbolic lattice $N^\vee(3)$ of same rank $4$, but discriminant $-3$.
The invariants are small enough to infer that $N^\vee(3)\cong U+A_2$.
Now the claim follows by applying the duality again (since $A_2^\vee(3)\cong A_2$).

To conclude, we return to the subgroup 
$$\bbF_3^{\,3}\cong H\subset A_{L^\perp}
\cong \bbF_p^{\,4}\times\bbF_3^{\,3}
$$
from \eqref{eq:H}.
The discriminant form on $H$ can be read off from \eqref{eq:N} as follows:
\begin{eqnarray}
\label{eq:q}
q_{L^\perp}|_H = p\cdot q_N = 
\begin{cases}
q_N & \text{ if $p\equiv 1$ mod } 3,\\
-q_N & \text{ if $p\equiv -1$ mod } 3,
\end{cases}
\end{eqnarray}
(the quadratic forms taking  values in  $\bbQ/2\bbZ$).
Recall that gluing $L$ to $L^\perp$ along $H$ requires exactly
that 
\[
q_L = - q_{L^\perp}|_H
\]
Independent of $p$, 
we already know that $q_L=q_N$ since $N=M(-1)$, see the proof of Lemma \ref{lem:L}.
Note that this asserts the case $p\equiv -1$ mod $3$ in \eqref{eq:q}.
Hence the other alternative, with $p\equiv 1$ mod $3$, can only persist
(for some $p$) if  $q_N=-q_N$.
But this is absurd -- for instance, it would imply that $N$ glues to itself to give an even unimodular lattice,
but this would have signature $(2,6)$, contradiction.
This completes the proof of Theorem \ref{thm2}.
\qed

\section{Proof of Theorem \ref{thm1}}
\label{s:pf1}

With these lattice theoretic preparations,
it is not hard to give a proof of Theorem \ref{thm1}.
Starting from a K3 surface $X$ containing 9 disjoint $A_2$ configurations
of smooth rational curves, we not only have the sublattice $L$ inside $\NS(X)$ from Lemma \ref{lem:L},
but we are also equipped with a vector $v$ of length $9$ inside the discriminant group $G$ of $A_2^9$
which in fact is integral, i.e.\ belongs to $L$ (see Remark \ref{rem:H}).
Explicitly, $v$ may be represented as
\[
v = \frac 13\sum_{i=1}^9 (C_i+2C_i')
\]
where the $C_i, C_i'$ are the smooth rational curves supporting the nine $A_2$ configurations
(up to exchanging the two curves).
Following classical theory (e.g.\ \cite{Miranda}), this divisor determines a triple covering of $X$
which we can use to our advantage.
Indeed, we can proceed exactly as in \cite[\S 5]{KS}, so we just give the rough outline of the construction:
\begin{enumerate} 
\item
blow up the intersection points $C_i\cap C_i'$ to get $\tilde X$;
\item
switch to the smooth triple covering $\tilde A$;
\item
minimalize to $A$ by first blowing down the strict transforms of the $C_i, C_i'$
and then those of the exceptional curves in $\tilde X$;
\item
check using the classification of algebraic surfaces that $A$ is an abelian surface.
\end{enumerate}
For each step, the arguments from \cite{KS} go through
-- regardless of the characteristic and of the question
whether $X$ is supersingular or not.
Automatically, the triple covering endows $A$ with an automorphism $\sigma$ of order $3$
such that $X$ can be recovered as minimal desingularization of the quotient $A/\langle\sigma\rangle$,
and this completes the proof.
(The following diagram of maps is only reproduced for the convenience of the reader.)
\qed

$$
\begin{array}{ccccc}
\tilde A & \to & \hat A &  \to & A\\
\downarrow &&&& \downarrow\\
\tilde X & \to & X & \to & A/\langle\sigma\rangle
\end{array}
$$

\section{Non-simple abelian surfaces with automorphism of order 3}

This section provides a proof of Proposition \ref{non-isogenous},
so we let $A$ be an abelian surface such that A is isogenous to $E_{1}\times E_{2}$
with elliptic curves $E_{i}$ ($i = 1, 2$). We assume that $E_{1}$ is not isogenous to $E_{2}$
and that $A$ admits an automorphism  $\sigma$  of order 3.

%

If $\sigma$ is fixed-point free, then it is clear that the quotient surface is
either abelian of hyperelliptic. 
It remains to consider the case where $\sigma$ has a fixed point on $A$, say $P$.
By our assumption, we have an exact sequence
$$
   0 \longrightarrow E \longrightarrow A \stackrel{f}{\longrightarrow} A/E \longrightarrow 0
$$
with an elliptic curve $E$. Let $F$ be a fiber of $f$ such that $F$ passes through 
the fixed point $P$. Then we have $\sigma(F) \cap F \ni P$. If $\sigma(F) \neq F$,
then $\sigma (F)$ would be a multi-section of $f$ and we would have an isogeny
from $\sigma (F)$ to $A/E$. Therefore, $A$ would be isogenous to $F \times A/E$
with $F$ isgenous to $A/E$, a contradiction to our assumption.
It follows that $F = \sigma (F)$. We put $f(F) = Q$. Then, we have $3F = f^{*}(3Q)$
and, by Riemann--Roch, $3Q$ is a very ample divisor on $A/E$. Therefore, the linear system $\mid 3F \mid$
gives the morphism $f$. Since $F$ is invariant under the action of $\sigma$, $\sigma$
acts on the vector space $L(3F)$. Therefore, $\sigma$ induces an action on $A/E$ fixing $Q$.

Suppose first that $\sigma$ does not act as the identity on $F$ nor on $A/E$.
Thus both elliptic curves admit an automorphism of order $3$,
but this curve is unique up to isomorphism (j-invariant zero),
so $F\cong A/E$, contradiction.


Suppose that $\sigma$ acts as identity on $A/E$.
Then we have a morphism
$A/\langle \sigma \rangle \longrightarrow A/E$. Therefore, we have the dimension 
$q(A/\langle \sigma \rangle) \geq 1$ of the Albanese variety of $A/\langle \sigma \rangle$. 
In particular, $A/\langle \sigma \rangle$ cannot be
birationally equivalent to a K3 surface. 

Suppose that $\sigma$ acts as identity on $F$. Since $F$ is non-singular, there exist
local coordinates $x$, $y$ such that $\sigma (x) = x$ and $\sigma (y) = \omega y$
with $\omega$, a cube root of unity (it may be 1). 
Applying the same argument to any fixed point of $\sigma$, 
we see that the quotient surface
$A/\langle \sigma \rangle$ is non-singular. Therefore, we have 
${\rm H}^{0}(A/\langle \sigma \rangle, \Omega_{A/\langle \sigma \rangle}^{1}) \cong
{\rm H}^{0}(A, \Omega_{A}^{1})^{\langle \sigma \rangle}$. 
Since we have a natural isomorphism
$m_{P}/m_{P}^{2} \cong {\rm H}^{0}(A, \Omega_{A}^{1})$ and 
$\dim (m_{P}/m_{P}^{2})^{\langle \sigma \rangle} \geq 1$, we obtain
$$\dim {\rm H}^{0}(A/\langle \sigma \rangle, \Omega_{A/\langle \sigma \rangle}^{1}) \geq 1,
$$
and again $A/\langle \sigma \rangle$ cannot be a K3 surface.
This concludes the proof of Proposition \ref{non-isogenous}.
\qed

\begin{corollary}
Let $A$ be an abelian surface with $p$-rank 1. Then, there exists no automorphism $\sigma$
of order 3 on $A$ such that $A/\langle \sigma \rangle$ is birationally equivalent to 
a K3 surface.
\end{corollary}

\begin{proof}
If $A$ is simple, we will show this corollary in the next section.
If $A$ is non-simple, $A$ is isogenous to a product $E_{1} \times E_{2}$
of two elliptic curves $E_{i}$ ($i = 1, 2$). Since the $p$-rank of $A$ is 1,
one of $E_{i}$'s is ordinary and the other is supersingular. Therefore,
$E_{1}$ is not isogenous to $E_{2}$. Hence, the result follows from 
Propositon \ref{non-isogenous}.
\end{proof}

\section{Endomorphism algebras of simple abelian surfaces}
\label{s:endo}

As before, we let $k$ be an algebraically closed field of characteristic $p > 0$.
We summerize here some results by Mumford on the endomorphism algebras
of simple abelian surfaces (\cite{M} Section 21, Theorem 2) over $k$.
Let $A$ be a simple abelian surface and let ${\rm End}(E)$ be the endomorphism ring of $A$.
We denote by $A_{n}$ the (reduced) $n$-torsion group of $A$.
We set $D = {\rm End}^{0}(E) = {\rm End}(E)\otimes_{{\bf Z}}{\bf Q}$. 
Then, $D$ is a central simple
division algebra. We denote by $K$ the center of $D$  and by $K_{0}$
the subfield of $K$ which is fixed by the Rosati involution.
We put $[D : K] = d^{2}$, $[K : {\bf Q}] = e$ and $[K_{0} : {\bf Q}] = e_{0}$.
We also put $S = \{ x \in D \mid x' = x\}$. It is known that
$\dim_{{\bf Q}} S$ is equal to the Picard number $\rho (A)$ of $A$.
We put $\eta = \frac{\dim_{{\bf Q}} S}{\dim_{{\bf Q}} D}$.
Then, Mumford gave the following table for the possible numerical
invariants of $D$.

\begin{table}[ht!]
\centering
\begin{tabular}{|c||c|c|c|r|} \hline 
Type &  $e$ & $d$ & $\eta$ &  char $p >0$ \\  \hline \hline
I & $e_{0}$ & $1$ & $1$ &  $e \mid 2$ \\
II & $e_{0}$ & $2$ & $\frac{3}{4}$ &  $2e \mid 2$ \\
III & $e_{0}$ & $2$ & $\frac{1}{4}$ &  $e \mid 2$ \\
IV & $2e_{0}$ & $d$ & $\frac{1}{2}$ & $e_{0}d \mid 2$\\ \hline
\end{tabular}
\end{table}

\noindent
Using this list, we get the following detailed list.

$$
\begin{array}{|c||c|c|c|c|c|c|}
\hline
\text{Type} & e & e_0 & d & \eta & \dim_{{\bf Q}}D & \rho(A)\\
\hline
\text{(I-i)} & 1 & 1 & 1 & 1 & 1 & 1\\
\hline
\text{(I-ii)} & 2 & 2 & 1 & 1 & 2 & 2\\
\hline
\text{(II)} & 1 & 1 & 2 &\frac 34 & 4 & 3\\
\hline
\text{(III-i)} & 1 & 1 & 2 &\frac 14 & 4 & 1\\
\hline
\text{(III-ii)} & 2 & 2 & 2 & \frac 14 & 8 & 2\\
\hline
\text{(IV-i)} & 2 & 1 & 1 & \frac 12 & 2 & 1\\
\hline
\text{(IV-ii)} & 2 & 1 & 2 & \frac 12 & 8 & 4\\
\hline
\text{(IV-iii)} & 4 & 2 & 1 & \frac 12 & 4 & 2\\
\hline
\end{array}
$$

%
%
%
%
%
%
%

We will show that the cases (III-ii), (IV-ii) and (IV-iii) cannot occur for a simple abelian surface $A$.
We denote the $p$-adic Tate module of $A$ by $T_{p}(A)$.
First we show the following lemma (see also \cite{M}, Section 19, Theorem 3).

\lemma\label{injective}
{Let $A$ be a simple abelian surface. Then, the natural homomorphism 
$$
{\rm End} (A)\otimes_{{\bf Z}}{\bf Z}_{p} \longrightarrow {\rm End}_{{\bf Z}_{p}}(T_{p}(A))
$$
is injective}.
\proof{In dimension 2, $A$ is supersingular if and only if the $p$-rank of $A$
is 0. Since $A$ is simple, $A$ is not supersingular  (cf. Oort \cite{Oo}). 
Therefore,
the rank of $T_{p}(A)$ is either 1 or 2. Since $A$ is simple, the kernel of a
non-zero endomorphism $f$ is a finite group scheme. Therefore, for a large
positive integer m, the induced homomorphism $f: A_{p^{m}} \longrightarrow A_{p^{m}}$
is not the zero-map. Therefore, the natural homomorphism ${\rm End} (A)\otimes_{{\bf Z}}{\bf Z}_{p} \longrightarrow {\rm End}_{{\bf Z}_{p}}(T_{p}(A))
$ is injective.
\qed}

\begin{lemma}
In the list above, the cases (III-ii), (IV-ii) and (IV-iii) cannot occur.
\end{lemma}

\begin{proof}
In cases (III-ii) and (IV-ii), we have 
$\dim_{{\bf Q}_{p}}D\otimes_{{\bf Q}}{\bf Q}_{p} = 8$. On the other hand,
$\dim_{{\bf Q}_{p}}{\rm End}_{{\bf Q}_{p}}(T_{p}(A)\otimes_{{\bf Z}_{p}} {\bf Q}_{p})$ is equal to 1 or 4,
according to the $p$-rank of $A$ = 1 or 2, which is impossible by Lemma \ref{injective}.
In case (IV-iii), since we have $\dim_{{\bf Q}}D = 4$, the $p$-rank of $A$ should
be 2 and $\dim_{{\bf Q}_{p}}{\rm End}_{{\bf Q}_{p}}(T_{p}(A)\otimes_{{\bf Z}_{p}} {\bf Q}_{p}) = 4$, and we have 
$D\otimes_{\bf Q}{{\bf Q}_{p}}\cong {\rm End}_{{\bf Q}_{p}}(T_{p}(A))\otimes_{{\bf Z}_{p}} {\bf Q}_{p}$. 
However, $D\otimes_{\bf Q}{{\bf Q}_{p}}$ is commutative and 
${\rm End}_{{\bf Q}_{p}}(T_{p}(A))\otimes_{{\bf Z}_{p}} {\bf Q}_{p}$ is non-commutative, a contradiction.
\end{proof}

By the list above, we have the following corollary.

\begin{corollary}
\label{Picard}
For simple abelian surfaces, we have $\rho (A) \leq 3$.
\end{corollary}

Note how this fits together  with the classic result of Shioda--Mitani \cite{SM}
that a complex abelian surface $A$ with $\rho(A)=4$ is isomorphic
to a product of elliptic curves.

\proposition\label{structure}
{Let $A$ be a simple abelian surface with an automorphism $\sigma$
of order 3. Then, the structure of the endomorphism algebra 
${\rm End}^{0}(A)$ of $A$ is one of the following.

(i) A division algebra over ${\bf Q}$ which contains ${\bf Q}(\sigma)$.

(ii) ${\rm End}^{0}(A) = {\bf Q}(\sigma)$ with 
$K_{0} = {\bf Q}$ and $K = {\bf Q}(\sigma)$.
}
\proof{This follows from the above classification of division algebras.
\qed}

\section{Simple abelian surfaces with automorphism of order 3}
\label{s:simple}

We shall now start working towards the proof of Theorem \ref{thm:simple}.
First comes the ordinarily condition imposed by automorphisms of order 3:

\begin{proposition}
Let $A$ be a simple abelian surface with an automorphism
$\sigma$ of order 3. Assume that $p \neq 3$ and that $\sigma$ is not a translation.
Then, $A$ is an ordinary abelian surface.
\end{proposition}

\begin{proof}
If the $p$-rank of $A$ is 0, then in case of dimension 2 $A$ is a supersingular
abelian surface as we have used above. 
Therefore, $A$ is not simple (cf. Oort \cite{Oo}).
Assume the $p$-rank of $A$ is equal to 1. Then, $T_{p}(A)$ has rank 1 over 
${\bf Z}_{p}$ and so ${\rm End}(T_{p}(A))\otimes_{{\bf Z}_{p}}{\bf Q}_{p}$
is 1-dimensional over ${\bf Q}_{p}$, which contradicts Proposition \ref{structure}
and Lemma \ref{injective}. Hence, the p-rank of $A$ is $2$, that is, $A$ is ordinary
as claimed.
\end{proof}

We use the Harder-Narashimhan theorem frequently.

\begin{theorem}
[Harder-Narashimhan \cite{HN}, Proposition 3.2.1]
Let $X$ be a nonsingular projective variety on which a finite group $G$
acts. Let $\ell$ be a prime number which is prime to both $p$ and the order of $G$.
Then, the \'etale cohomology ${\rm H}^{i}(X/G, {\bf Q}_{\ell})$ is isomorphic to
the subspace ${\rm H}^{i}(X, {\bf Q}_{\ell})^{G}$ of $G$-invariants in 
${\rm H}^{i}(X, {\bf Q}_{\ell})$:
$$
{\rm H}^{i}(X/G, {\bf Q}_{\ell}) \cong {\rm H}^{i}(X, {\bf Q}_{\ell})^{G}.
$$
\end{theorem}

Applied to quotients of projective surfaces, we obtain the following:

\begin{lemma}
\label{injectiveness}
Let $X$ be a nonsingular projective surface on which a finite group $G$
acts. Let $\ell$ be a prime number which is prime to both $p$ and the order of $G$.
Moreover, assume $G$ has only isolated fixed points, and
let $\varphi : Y \longrightarrow X/G$ be a minimal resolution of $X/G$.
Then, we have an isomorpism
$$
\varphi^{*} :{\rm H}^{1}(X/G, {\bf Q}_{\ell}) \cong {\rm H}^{1}(Y, {\bf Q}_{\ell})
$$
and an injective homomorphism  
$$
\varphi^{*} :{\rm H}^{2}(X/G, {\bf Q}_{\ell}) \longrightarrow {\rm H}^{2}(Y, {\bf Q}_{\ell}).
$$ 
\end{lemma}

\begin{proof}
Let $W$ be the set of singular points of $X/G$, and $E$ be the exceptional divisor
of $\varphi$ on $Y$. Then, we have an isomorphism 
$$
\varphi|_{Y\setminus E} : Y \setminus E \longrightarrow X/G \setminus W.
$$
Therefore, we have an isomorphism ${\rm H}_c^{i}(A/G \setminus W, {\bf Q}_{\ell}) \cong
{\rm H}_c^{i}(Y \setminus E, {\bf Q}_{\ell})$. 
There is a commutative diagram of long exact sequences 
of \'etale cohomology groups with compact support 
whose coefficients are in ${\bf Q}_{\ell}$ (cf. Milne \cite{Mi}):
$$
\begin{array}{ccccccccc}
\rightarrow &{\rm H}_c^{i -1}(W, {\bf Q}_{\ell}) &\rightarrow &
{\rm H}_c^{i}(A/G \setminus W, {\bf Q}_{\ell})&
\rightarrow &{\rm H}_c^{i}(A/G, {\bf Q}_{\ell})&
\rightarrow &{\rm H}_c^{i }(W, {\bf Q}_{\ell})& \rightarrow\\
  & \downarrow & & \downarrow & & \downarrow & & \downarrow   & \\
\rightarrow &{\rm H}_c^{i -1}(E, {\bf Q}_{\ell}) &\rightarrow &
{\rm H}_c^{i}(Y \setminus E, {\bf Q}_{\ell})&
\rightarrow &{\rm H}_c^{i}(Y, {\bf Q}_{\ell})&
\rightarrow &{\rm H}_c^{i }(E, {\bf Q}_{\ell})& \rightarrow
\end{array}
$$
The singularities of $A/G$ are rational by \cite[p.\ 149]{Pinkham}
(which assumes characteristic zero,
but the trace argument works
in characteristic $p$ as long as the order of $G$ is prime to $p$).
Hence $E$ consists
of trees of ${\bf P}^{1}$'s. Therefore, we have
${\rm H}_c^{1}(E, {\bf Q}_{\ell}) = 0$. We also have
$$
\begin{array}{l}
{\rm H}_c^{1}(W, {\bf Q}_{\ell}) = {\rm H}_c^{2}(W, {\bf Q}_{\ell})= 0, \\
{\rm H}_c^{i}(A/G, {\bf Q}_{\ell}) \cong {\rm H}^{i}(A/G, {\bf Q}_{\ell}) 
\quad (i = 1, 2), \\
{\rm H}_c^{i}(Y, {\bf Q}_{\ell}) \cong {\rm H}^{i}(Y, {\bf Q}_{\ell}) 
\quad (i = 1, 2).
\end{array}
$$
The results follow from these facts. 
\end{proof}

We will also need the following helpful property.

\begin{lemma}
\label{ruled}
Let $A$ be an abelian surface, and $C$ be a nonsingular complete curve
of genus $g \geq 2$.  Then, there exists no non-trivial rational map from $A$ to $C$.
\end{lemma}

\begin{proof}
Suppose there exists a non-trivial rational map $f : A \longrightarrow C$. 
Then, by composition, there exists a homomorphism from $A$ to the Jacobian variety $J(C)$
of $C$. Since the homomorphism factors through $C$, it is absurd.
\end{proof}

\begin{lemma}
\label{finite}
Let $A$ be a simple abelian surface with an automorphism
$\sigma$ of order 3. Assume that $p \neq 3$ and that $\sigma$ is not a translation.
Then, $\sigma$ has at least one fixed point and the fixed locus consists
of finitely many points.
\end{lemma}

\begin{proof}
If $\sigma$ is fixed-point-free, then the quotient surface 
$A/\langle \sigma \rangle$ is either an abelian
surface or a hyperelliptic surface. If it is an Abelian surface, $\sigma$ must 
be a translation,
which contradicts our assumption. If it is a hyperelliptic surface, then
the Albanese variety $Alb(A/\langle \sigma \rangle)$ is an elliptic curve and 
we have a surjective morphism
from $A$ to $Alb(A/\langle \sigma \rangle)$, which contradicts our  assumption 
that $A$ is simple. Now, we may choose a fixed point of $\sigma$ 
as the zero point of $A$. Then, $\sigma$ is a homomorphism.
Since $A$ is simple, the kernel of the homomorphism $\sigma - id_{A}$
is finite. Therefore, the fixed locus of $\sigma$ is a finite set.
\end{proof}

We denote by $\omega$ a primitive cube root of unity.

\begin{lemma}
\label{eigenvalue}
Let $A$ be a simple abelian surface with an automorphism
$\sigma$ of order 3. Assume that $p \neq 3$ and that $\sigma$ is not a translation.
Then, the eigenvalues of $\sigma$ on the \'etale cohomology group 
${\rm H}^{1}(A, {\bf Q}_{\ell})$ are given by $\omega$, $\omega$,
$\omega^{2}$ and $\omega^{2}$.
\end{lemma}

\begin{proof}
Since $\sigma^{3} = id_{A}$ and $\sigma - id_{A}$ is an isogeny, 
we see $\sigma^{2} + \sigma + id_{A} = 0$.
Therefore, the minimal polynomial of $\sigma$ is $x^{2} + x + 1$ 
(cf. Mumford \cite{M}, Section 19, Theorem 4). Therefore, the possibilities of 
the eigenvalues of $\sigma$ on ${\rm H}^{1}(A, {\bf Q}_{\ell})$ are the following.

Case (i) 1, 1, 1, 1.

Case (ii) 1, 1, $\omega$, $\omega^{2}$.

Case (iii) $\omega$, $\omega$, $\omega^{2}$, $\omega^{2}$.

Suppose Case (i). Then, 
since ${\rm H}^{i}(A, {\bf Q}_{\ell}) \cong \wedge^{i}{\rm H}^{1}(A, {\bf Q}_{\ell})$,
we see that all the eigenvalues of $\sigma$ on ${\rm H}^{*}(A, {\bf Q}_{\ell})$
are 1. Hence, the alternating sum of traces of $\sigma$ on ${\rm H}^{*}(A, {\bf Q}_{\ell})$
is equal to 0. Hence, by the Lefschetz trace formula, $\sigma$ is fixed-point-free on $A$,
which contradicts Lemma \ref{finite}.
Therefore, Case (i) is excluded.

Now, we denote by $Y \longrightarrow A/\langle \sigma \rangle$ 
a resolution of singularities of $A/\langle \sigma \rangle$.
Then, by Lemma \ref{injectiveness}, we have an isomorphism
${\rm H}^{1}(A/\langle \sigma \rangle, {\bf Q}_{\ell}) \cong {\rm H}^{1}(Y, {\bf Q}_{\ell})$. 
and we have 
$\dim {\rm H}^{1}(Y, {\bf Q}_{\ell}) = \dim {\rm H}^{1}(A, {\bf Q}_{\ell})^{\langle \sigma \rangle}$.

Suppose Case (ii). Then we have 
${\rm H}^{1}(Y, {\bf Q}_{\ell}) = 
\dim {\rm H}^{1}(A, {\bf Q}_{\ell})^{\langle \sigma \rangle} = 2$.
Therefore, the dimension $q(Y)$ of the Albanese variety of $Y$ is equal to 1.
Therefore, we have a surjective homomorphism from $A$ to the Albanese variety 
(an elliptic curve), which contradicts the assumption that $A$ is simple.

Hence, we conclude that Case (iii) holds.
\end{proof}

\begin{corollary}
Let $A$ be a simple abelian surface with an automorphism
$\sigma$ of order 3. Assume that $p \neq 3$ and that $\sigma$ is not a translation.
Then, the number of fixed points of $\sigma$ is equal to 9.
\end{corollary}

\begin{proof}
Since $A$ is simple, the fixed loci of $\sigma$ are isolated. 
Since ${\rm H}^{2}(A, {\bf Q}_{\ell}) \cong \wedge^{2}{\rm H}^{1}(A, {\bf Q}_{\ell})$,
the eigenvalues of $\sigma$ on ${\rm H}^{2}(A, {\bf Q}_{\ell})$ are given by
\begin{eqnarray}
\label{eq:ev}
  1, 1, 1, 1, \omega,  \omega^{2},
\end{eqnarray}
and on ${\rm H}^{1}(A, {\bf Q}_{\ell})$ they are the same as on ${\rm H}^{3}(A, {\bf Q}_{\ell})$.
By the Lefschetz trace formula, we see that the number of fixed points
is equal to 9.
\end{proof}

We are now ready to prove Theorem \ref{thm:simple}.
Let $Y \longrightarrow A/\langle \sigma \rangle$ be a resolution of singularities 
of $A/\langle \sigma \rangle$.
Since we have a separable dominating rational map from $A$ to $Y$, we see 
$$0 =\kappa (A) \geq \kappa (Y).
$$
 By the Enriques--Kodaira classification
 (extended to positive characteristic by Bombieri--Mumford), $Y$ is a K3 surface,  an Abelian surface, 
a hyperelliptic surface, an Enriques surface or a ruled surface.
If $Y$ is an Abelian surface, the rational map from $A$ to $Y$ is
a homomorphism. Therefore, $\sigma$ must coincide with a translation, which contradicts
our assumption. If $Y$ is a ruled surface with $q(Y) \geq 2$. Then, we have a rational
map from $A$ to $Y$. Therefore, we have a rational map from $A$ to the base curve of $Y$,
which is a curve of genus $\geq 2$. A contradiction to Lemma \ref{ruled}.
If $Y$ is either hyperelliptic or ruled with $q(Y) = 1$, then we have a homomorphism
from $A$ to an elliptic curve --  which contradicts that $A$ is simple.
If $Y$ is either rational or Enriques, then we have an inclusion
$$
       {\rm H}^{2}(A, {\bf Q}_{\ell})^{\langle \sigma \rangle}\cong 
       {\rm H}^{i}(A/\langle \sigma \rangle, {\bf Q}_{\ell})\hookrightarrow {\rm H}^{2}(Y, {\bf Q}_{\ell}).
$$
Since $Y$ is supersingular in the sense of Shioda, that is, ${\rm H}^{2}(Y, {\bf Q}_{\ell})$
is generated by algebraic cycles, we see that 
${\rm H}^{2}(A, {\bf Q}_{\ell})^{\langle \sigma \rangle}$ is generated by algebraic cycles.
Since $\dim {\rm H}^{2}(A, {\bf Q}_{\ell})^{\langle \sigma \rangle} = 4$ by \eqref{eq:ev}, we see the Picard number 
$\rho (A) \geq 4$, which contradicts Corollary \ref{Picard}. Hence, $A$ is a K3 surface.
This completes the proof of Theorem \ref{thm:simple}.
\qed

\medskip

Summarizing these results, we have the following corollary.

\begin{corollary}
Let $A$ be a simple ordinary abelian surface with an automorphism
$\sigma$ of order 3. Assume $p \neq 3$ and $\sigma$ is not a translation.
Then, $A/ \langle \sigma \rangle$ has just 9 $A_{2}$-rational double points
as singular points, and the minimal resolution is a K3 surface with $\rho=19$.
\end{corollary}

\section{Explicit quotients of simple abelian surfaces}

Exhibiting explicit simple abelian surfaces turns out to be a non-trivial problem
in positive characteristic -- especially in characteristic two.
For this reason, we decided to include families of K3 surfaces with nine cusps
in any characteristic $p\neq 3$ 
such that the covering abelian surfaces are generically simple.

To explain the approach, we recall from \cite{Barth2}
that  complex tori $A$ with an automorphism $\sigma$ of order $3$
come in a two-dimensional analytic family such that generically
\[
\NS(A)=A_2, \;\;\; T_A = M_0 = U+A_2(-1).
\]
Algebraic subfamilies are obtained by enhancing the N\'eron--Severi lattice by a positive vector $H$ from $M_0$;
the generic N\'eron--Severi lattice is thus promoted 
to the primitive closure $N$ of $\bbZ H+A_2$ inside ${\rm H}^2(A,\bbZ)\cong U^3$.
The very general member of the resulting one-dimensional family
is simple if and only if $N$ does not represent zero non-trivially.
In \cite{Barth2}, an abstract example with $H^2=12$ is worked out;
in contrast we will work out an explicit example with $H^2=10$,
though admittedly, it is fully  explicit only on the K3 side 
(which can be used to recover $A$ as explained in Section \ref{s:pf1}).
To this end, take $H\in U\subset M_0$ with $H^2=10$ and postulate that $H\in\NS(A)$.
Then this determines a one-dimensional family of abelian surfaces $A$
with an automorphism $\sigma$ of order $3$ 
such that generically 
$$
\NS(A)=\bbZ H+A_2 \;\;\; \text{ and } \;\;\; T_A=\bbZ(-10)+A_2(-1).
$$

Consider the family of K3 surfaces $X$ which arise as minimal resolutions of the quotients $A/\langle\sigma\rangle$.
Then these always have $L\subset\NS(X)$,
and following \cite{Barth2}, the sublattice $M_0$ pushes down to $M=U(3)+A_2(-1)$
(the lattice from Lemma \ref{lem:L}).
The algebraic enhancement means that $H$ induces a positive vector $v$ of square $v^2=30$ in $\NS(X)$,
such that generically 
\begin{eqnarray}
\label{eq:T}
\;\;\; \;\;T_X = (v^\perp\subset M_0) =  \bbZ(-30)+A_2(-1), \;\;\; \NS(X) \supset \bbZ v+ L
\end{eqnarray}
where the last inclusion has index $3$ for discriminant reasons.

\begin{lemma}
Generically, one has 
\begin{eqnarray}
\label{eq:NS}
\NS(X) = U + 2 E_6 + A_4 + A_1.
\end{eqnarray}
\end{lemma}

\begin{proof}
By \cite{Nikulin} suffices to verify that the discriminant forms of N\'eron--Severi lattice and transcendental lattice 
generically agree up to sign;
i.e.\ for $T_X$ from \eqref{eq:T} 
and $\NS(X)$ as in \eqref{eq:NS},
we have $q_\NS = - q_T$
which is readily verified.
\end{proof}

The above representation of $\NS(X)$ is very convenient because it implies by standard arguments 
(see \cite{SSh}, for instance)
that $X$ admits an elliptic fibration such that generically there is only a single section
(so most of $\NS$ is captured by  fibre components).
One can use this as a starting point to work out the following family of elliptic K3 surfaces
with 9 cusps, given by in affine Weierstrass form with parameter $\lambda$:
\begin{eqnarray*}
y^2 + (\lambda+1)txy & = & x^3+t(3t^2-t(\lambda^2-4\lambda+1)+3\lambda^2)x^2\\
&&+3t^2(t-1)^2(t+\lambda^3)(t+\lambda)x+t^3(t-1)^4(t+\lambda^3)^2
\end{eqnarray*}

\begin{proposition}
\label{prop:19}
In any characteristic $\neq 3,5$,
the  family $\mathcal X$ has generically $\rho(\mathcal X)=19$
and $\NS(\mathcal X)=U + 2 E_6 + A_4 + A_1$.
\end{proposition}

Before coming to the proof of the proposition, we note 
that we can recover the family of covering abelian surfaces
from $X$ 
by the geometric argument from Section \ref{s:pf1}.
In particular, Proposition \ref{prop:19} implies the following:

\begin{corollary}
\label{cor:simple35}
The covering abelian surfaces are generically simple
 in any characteristic $\neq 3, 5$.
\end{corollary}

\begin{proof}[Proof of Proposition \ref{prop:19}]
We first prove that the family $\mathcal X$ is non-isotrivial.
To this end, we use that the discriminant $\Delta$ of the above elliptic fibration 
obviously varies with $\lambda$
-- and so does the j-invariant.
Hence, if the family were isotrivial, i.e.\ almost all members isomorphic to a single K3 surface $X_0$,
then $X_0$ would admit infinitely many non-isomorphic elliptic fibrations.
Over fields of characteristic $\neq 2$, this is ruled out by work of Sterk \cite{Sterk} 
and Lieblich--Maulik \cite{LM}.

In characteristic $2$, it suffices by \cite{MM} to exhibit two non-isomorphic smooth specializations within $\mathcal X$.
For this purpose, we endow special members of the family with a suitable section as follows.
We start by arguing in characteristic zero with a root $\alpha$ of  $\alpha^3-2\alpha^2-3\alpha+9$.
Let $L=\mathbb Q(\alpha)$.
Then the special member $X$ of the family $\mathcal X$ at $\lambda=\alpha$ over $L$
admits a section of height $29/30$ with $x$-coordinate
$-t(t-1)^2(t+\alpha^3)/\alpha$.
It follows that $X$ is a singular K3 surface of discriminant $-87$.
Note that $X$ has smooth reduction $X_2$ over $\mathbb F_8$.
Arguing as in \cite[proof of Claim 10.3]{S-mod2},
one finds that $X_2$ is ordinary (i.e. $\rho(X_2)=20$ 
with $\NS(X_2)$ of the same discriminant $-87$).

To compare with another member of the family, we work exclusively in characteristic $2$ (to limit the complexity).
Let $\beta\in\bbF_{256}$ be a root of $\beta^8+\beta^5+\beta^4+\beta^3+1$
and consider the member $X'$ of the family $\mathcal X$ at $\lambda=\beta$.
One finds ($\beta$ by requiring) that $X'$ admits a section of height $61/30$;
its $x$-coordinate is $t(t+1)(t+\beta^3)(t^2+t+1)/(t+\beta^6+\beta^5+\beta^4+\beta)^2$.
As before, this implies $\rho(X')=20$ and $\det\NS(X')=-183$.
In particular, $X_2\not\cong X'$, so the family $\mathcal X$ is non-isotrivial in characteristic $2$ as claimed.

%

We proceed by proving the statement about the generic Picard number -- which clearly satisfies $\rho\geq 19$.
Since we have a one-dimensional family, the only alternative to $\rho=19$ is $\rho=22$, and only in positive characteristic
(because K3 surfaces with $\rho=20$ do not move in a family (just like over $\bbC$),
and $\rho=21$ is impossible, see \cite{Artin}). 
So let us assume that $\rho=22$ and char$(k)=p>0$.
By \cite{Ogus}, there is a unique supersingular K3 surface of Artin invariant $\sigma=1$,
so we would require $\sigma\geq 2$.
For $p\neq 2$ we can argue along the same lines as in Section \ref{s:pf2}:
since $N_0=U + 2 E_6 + A_4 + A_1$ embeds primitively into $\NS=\Lambda=\Lambda_{p,\sigma}$,
but the discriminants  $d(N_0)=90$ and $d(\Lambda_{p,\sigma})=-p^{2\sigma}$ are relatively prime,
the discriminant group of $A_\Lambda$ would be fully supported on $A_{N_0^\perp}$.
The length of this group is bounded by the rank of $N_0^\perp$, i.e. $2\sigma\leq 3$,
contradiction.
To complete the argument, we appeal to the non-supersingular members $X_2$ or $X'$ of the family $\mathcal X$
in characteristic $2$ 
which we have already used  above to prove the non-isotriviality of the family $\mathcal X$.

Having shown that generically $\rho=19$, it remains to prove that 
the N\'eron--Severi lattice generically assumes the given shape, 
i.e.\ $\NS(\mathcal X)\cong N_0$.
By inspection of the discriminant $d(N_0)=90$, $\NS(\mathcal X)$ would otherwise have to be an index $3$ overlattice $N_0$.
But then one verifies that the discriminant group $A_{N_0}$ does not contain any non-zero isotropic elements,
so there is no integral overlattice at all.
\end{proof}

%

\subsection{Comments on characteristic $5$}
\label{s:5}


In characteristic $5$, the full family $\mathcal X$ turns out to be supersingular
(quite remarkably, without the singular fibers degenerating).
For instance, the generic fibre, base changed to $k(\sqrt \lambda)$, admits a section
of height $5/6$
with $x$-coordinate $-t(t+\lambda^3)(t+1/\lambda)$.

In order to work out an analogue of Corollary \ref{cor:simple35} in characteristic $5$,
one can apply the same procedures as above to an initial positive vector $H\in U\subset M_0$
with $H^2=4$.
Along similar lines, 
this leads to the following (non-isotrivial) family of K3 surfaces over $\mathbb Q$
\[
\mathcal Y: \;\;\; y^2+t^2(t-1)^2y = \mu(x^3-3t^3(t-1)^2x-2t^4(t-1)^3).
\]
One shows as before that generically $\rho(\mathcal Y)=19$
and $\NS(\mathcal Y)=U+2E_6+D_5$ outside characteristics $2,3$,
so one obtains simple abelian surfaces as in Corollary \ref{cor:simple35}.

Specifically for characteristic $5$, one can work, for instance, with the
special member $Y$ at $\mu=8/17$ admitting a $\bbQ(\sqrt{17})$-rational section of 
height $17/12$ with $x$-coordinate $t^2(t-1)(t-19)/18$.
The reduction to characteristic $5$ is seen to be ordinary.

In comparison, in characteristic $2$, the full family $\mathcal Y$ turns out to be supersingular.

%

\end{document}